\documentclass[11pt]{amsart}
\usepackage[T1]{fontenc}
\usepackage{amsfonts}
\usepackage[square, comma,  sort&compress, numbers]{natbib}
\usepackage{color}

\makeatletter
\theoremstyle{plain}
\newtheorem{thm}{Theorem}[section]
\theoremstyle{plain}

\numberwithin{equation}{section}
\numberwithin{figure}{section} 
\theoremstyle{plain}
\newtheorem{cor}[thm]{Corollary}
\theoremstyle{plain}
\newtheorem{lem}[thm]{Lemma}
\theoremstyle{plain}

\theoremstyle{plain}
\newtheorem{rem}[thm]{Remark}
\theoremstyle{plain}

\theoremstyle{plain}

\theoremstyle{plain}

\theoremstyle{plain}

\input{tcilatex}

\begin{document}
\title{On The Spaces of Fibonacci Difference Null and Convergent Sequences}
\date{}
\author{Met\.{I}n Ba\c{s}ar\i r}
\address{Sakarya \"{U}n{\tiny \.{I}}vers{\tiny \.{I}}tes{\tiny \.{I}},
Fen-Edeb{\tiny \.{I}}yat Fak\"{u}ltes{\tiny \.{I}}, Matemat{\tiny \.{I}}k B%
\"{o}l\"{u}m\"{u}, Esentepe Kamp\"{u}s\"{u}, Sakarya-54187, T\"{u}rk{\tiny 
\.{I}}ye}
\email{basarir@sakarya.edu.tr}
\author{Feyz\.{I} Ba\c{s}ar}
\address{Fat{\tiny \.{I}}h \"{U}n{\tiny \.{I}}vers{\tiny \.{I}}tes{\tiny 
\.{I}}, Fen-Edeb{\tiny \.{I}}yat Fak\"{u}ltes{\tiny \.{I}}, Matemat{\tiny 
\.{I}}k B\"{o}l\"{u}m\"{u}, B\"{u}y\"{u}k\c{c}ekmece Kamp\"{u}s\"{u}, \.{I}%
stanbul-34500, T\"{u}rk{\tiny \.{I}}ye}
\email{f.basar@fatih.edu.tr, feyzibasar@gmail.com}
\author{Emrah Evren Kara}
\address{D\"{u}zce \"{U}n{\tiny \.{I}}vers{\tiny \.{I}}tes{\tiny \.{I}},
Fen-Edeb{\tiny \.{I}}yat Fak\"{u}ltes{\tiny \.{I}}, Matemat{\tiny \.{I}}k B%
\"{o}l\"{u}m\"{u}, Konuralp Yerle\c{s}kesi, D\"{u}zce-81620, T\"{u}rk{\tiny 
\.{I}}ye}
\email{eevrenkara@hotmail.com}
\keywords{Fibonacci numbers, sequence spaces, difference matrix, alpha-,
beta-, gamma-duals, matrix transformations.}
\subjclass[2000]{ 11B39, 46A45, 46B45.}

\begin{abstract}
In the present paper, by using the band matrix $\widehat{F}$ defined by the
Fibonacci sequence, we introduce the sequence sequence spaces $c_{0}(%
\widehat{F})$ and $c(\widehat {F})$. Also, we give some inclusion relations
and construct the bases of the spaces $c_{0}(\widehat{F})$ and $c(\widehat{F}%
)$. Finally, we compute the alpha-, beta-, gamma-duals of these spaces and
characterize the classes $(c_{0}(\widehat{F}),X)$ and $(c(\widehat{F}),X)$
for certain choice of the sequence space $X$.
\end{abstract}

\maketitle

\section{Introduction}

By $\mathbb{N}$ and $\mathbb{R}$, we denote the sets of all natural and real
numbers, respectively. Let $\omega$ be the vector space of all real
sequences. Any vector subspace of $\omega$ is called a \textit{sequence space%
}. Let $\ell_{\infty}$, $c$, $c_{0}$ and $\ell_{p}$ denote the classes of
all bounded, convergent, null and absolutely $p$-summable sequences,
respectively; where $1\leq p<\infty$. Moreover, we write $bs$ and $cs$ for
the spaces of all bounded and convergent series, respectively. Also, we use
the conventions that $e=(1,1,1,\ldots )$ and $e^{(n)}$ is the sequence whose
only non-zero term is $1$ in the $n^{\text{th}}$ place for each $n\in\mathbb{%
N}$.

Let $\lambda$ and $\mu$ be two sequence spaces and $A=(a_{nk})$ be an
infinite matrix of real numbers $a_{nk}$, where $n,k\in\mathbb{N}$. Then, we
say that $A$ defines a \textit{matrix transformation} from $\lambda$ into $%
\mu$ and we denote it by writing $A:\lambda \rightarrow \mu,$ if for every
sequence $x=(x_{k})\in \lambda$ the sequence $Ax=\{A_{n}(x)\}$, the $A$%
-transform of $x$, is in $\mu$; where 
\begin{eqnarray}  \label{eq11}
A_{n}(x)=\sum_{k}a_{nk}x_{k}~\text{ for each }~n\in \mathbb{N}.
\end{eqnarray}
For simplicity in notation, here and in what follows, the summation without
limits runs from $0$ to $\infty$. By $(\lambda,\mu)$, we denote the class of
all matrices $A$ such that $A:\lambda \rightarrow \mu$. Thus, $%
A\in(\lambda,\mu)$ if and only if the series on the right side of (\ref{eq11}%
) converges for each $n\in\mathbb{N}$ and every $x\in\lambda$, and we have $%
Ax\in \mu$ for all $x\in \lambda$. Also, we write $A_{n}=(a_{nk})_{k\in%
\mathbb{N}}$ for the sequence in the $n$-th row of $A$.

The \textit{matrix domain} $\lambda_{A}$ of an infinite matrix $A$ in a
sequence space $\lambda$ is defined by 
\begin{eqnarray}  \label{eq12}
\lambda_{A}=\left\{x=(x_{k})\in\omega:Ax\in \lambda\right\}
\end{eqnarray}
which is a sequence space.

By using the matrix domain of a triangle infinite matrix, so many sequence
spaces have recently been defined by several authors, see for instance \cite%
{1,2,3,4,5,6,7}. In the literature, the matrix domain $\lambda_{\Delta}$ is
called the \textit{difference sequence space} whenever $\lambda$ is a normed
or paranormed sequence space, where $\Delta$ denotes the backward difference
matrix $\Delta=(\Delta_{nk})$ and $\Delta^{\prime }=(\Delta^{\prime }_{nk})$
denotes the transpose of the matrix $\Delta$, the forward difference matrix,
which are defined by 
\begin{eqnarray*}
\Delta_{nk}&=&\left\{%
\begin{array}{ccl}
(-1)^{n-k} & , & n-1\leq k\leq n, \\ 
0 & , & 0\leq k<n-1\text{ or\ }k>n,%
\end{array}%
\right. \\
\Delta^{\prime }_{nk}&=&\left\{%
\begin{array}{ccl}
(-1)^{n-k} & , & n\leq k\leq n+1, \\ 
0 & , & 0\leq k<n\text{ or\ }k>n+1%
\end{array}
\right.
\end{eqnarray*}
for all $k,n\in\mathbb{N}$; respectively. The notion of difference sequence
spaces was introduced by K\i zmaz \cite{8}, who defined the sequence spaces 
\begin{eqnarray*}
X(\Delta)=\left\{x=(x_{k})\in \omega:(x_{k}-x_{k+1})\in X\right\}
\end{eqnarray*}
for $X=\ell_{\infty}$, $c$ and $c_{0}$. The difference space $bv_{p}$,
consisting of all sequences $(x_{k})$ such that $(x_{k}-x_{k-1})$ is in the
sequence space $\ell_{p}$, was studied in the case $0<p<1$ by Altay and Ba%
\c{s}ar \cite{9} and in the case $1\leq p\leq\infty$ by Ba\c{s}ar and Altay 
\cite{10}, and \c{C}olak et al. \cite{11}. Kiri\c{s}\c{c}i and Ba\c{s}ar 
\cite{4} have been introduced and studied the generalized difference
sequence spaces 
\begin{eqnarray*}
\widehat{X}=\left \{x=(x_{k})\in \omega:B(r,s)x\in X\right\},
\end{eqnarray*}
where $X$ denotes any of the spaces $\ell_{\infty}$, $\ell_{p}$, $c$ and $%
c_{0}$, $1\leq p<\infty$, and $B(r,s)x=(sx_{k-1}+rx_{k})$ with $r,s\in%
\mathbb{R}\setminus\{0\}$. Following Kiri\c{s}\c{c}i and Ba\c{s}ar [4], S%
\"{o}nmez \cite{17} have been examined the sequence space $X(B)$ as the set
of all sequences whose $B(r,s,t)$-transforms are in the space $X\in\{
\ell_{\infty},\ell_{p},c,c_{0}\}$, where $B(r,s,t)$ denotes the triple band
matrix $B(r,s,t)=\{b_{nk}(r,s,t)\}$ defined by 
\begin{eqnarray*}
b_{nk}(r,s,t)=\left\{%
\begin{array}{ccl}
r & , & n=k \\ 
s & , & n=k+1 \\ 
t & , & n=k+2 \\ 
0 & , & \text{otherwise}%
\end{array}
\right.
\end{eqnarray*}
for all $k,n\in\mathbb{N}$ and $r,s,t\in\mathbb{R}\setminus\{0\}$. Also in 
\cite{18,19,20,21,22,23,25,26,27}, authors studied certain difference
sequence spaces. Furthermore, quite recently, Kara \cite{28} has defined the
Fibonacci difference matrix $\widehat{F}$ by means of the Fibonacci sequence 
$(f_{n})_{n\in\mathbb{N}}$ and introduced the new sequence spaces $\ell_{p}(%
\widehat{F})$ and $\ell_{\infty}(\widehat{F})$ which are derived by the
matrix domain of $\widehat{F}$ in the sequence spaces $\ell_{p}$ and $%
\ell_{\infty}$, respectively; where $1\leq p<\infty$.

In this paper, we introduce the sequence spaces $c_{0}(\widehat{F})$ and $c(%
\widehat{F})$ by using the Fibonacci difference matrix $\widehat{F}$. The
rest of this paper is organized, as follows:

In Section 2, we give some notations and basic concepts. In Section 3, we
introduce sequence spaces $c_{0}(\widehat{F})$ and $c(\widehat{F})$, and
establish some inclusion relations. Also, we construct the bases of these
spaces. In Section 4, the alpha-, beta-, gamma-duals of the spaces $c_{0}(%
\widehat{F})$ and $c(\widehat{F})$ are determined and the classes $(c_{0}(%
\widehat{F}),X)$ and $(c(\widehat{F}),X)$ of matrix transformations are
characterized, where $X$ denotes any of the spaces $\ell_{\infty}$, $f$, $c$%
, $f_{0}$, $c_{0}$, $bs$, $fs$ and $\ell_{1}$.

\section{Preliminaries}

A \textit{$B$-space} is a complete normed space. A topological sequence
space in which all coordinate functionals $\pi_{k}$, $\pi_{k}(x)=x_{k}$, are
continuous is called a \textit{$K$-space}. A \textit{$BK$-space} is defined
as a $K$-space which is also a $B$-space, that is, a $BK$-space is a Banach
space with continuous coordinates. For example, the space $\ell_{p}$ is $BK$%
-space with $\Vert x\Vert_{p}=\left(\sum_{k}\left\vert
x_{k}\right\vert^{p}\right)^{1/p}$ and $c_{0}$, $c$ and $\ell_{\infty}$ are $%
BK$-spaces with $\Vert x\Vert_{\infty}=\sup_{k\in\mathbb{N}}\vert x_{k}\vert$%
, where $1\leq p<\infty$. The sequence space $\lambda$ is said to be \textit{%
solid} (cf. \cite[p. 48]{kg}) if and only if 
\begin{eqnarray*}
\widetilde{\lambda}:=\{(u_{k})\in \omega :\exists(x_{k})\in\lambda \text{
such that } |u_{k}|\leq|x_{k}|~\text{ for all }~k\in\mathbb{N}%
\}\subset\lambda.
\end{eqnarray*}

A sequence $(b_{n})$ in a normed space $X$ is called a \textit{Schauder basis%
} for $X$ if for every $x\in X$ there is a unique sequence $(\alpha_{n})$ of
scalars such that $x=\sum_{n}\alpha _{n}b_{n}$, i.e., $\lim_{m}\left \Vert
x-\sum_{n=0}^{m}\alpha_{n} b_{n}\right \Vert =0.$

The alpha-, beta- and gamma-duals $\lambda^{\alpha}$, $\lambda^{\beta}$ and $%
\lambda^{\gamma}$ of a sequence space $\lambda$ are respectively defined by 
\begin{eqnarray*}
\lambda^{\alpha}&=&\left \{a=(a_{k})\in\omega:ax=(a_{k}x_{k})\in \ell_{1}%
\text{ for all }x=(x_{k})\in\lambda\right\}, \\
\lambda^{\beta}&=&\left\{a=(a_{k})\in \omega:ax=(a_{k}x_{k})\in cs\text{ for
all }x=(x_{k})\in\lambda\right\}, \\
\lambda^{\gamma}&=&\left\{a=(a_{k})\in \omega:ax=(a_{k}x_{k})\in bs\text{
for all }x=(x_{k})\in\lambda\right\}.
\end{eqnarray*}

The sequence $(f_{n})$ of Fibonacci numbers defined by the linear recurrence
equalities 
\begin{eqnarray*}
f_{0}=f_{1}=1~\text{ and }~f_{n}=f_{n-1}+f_{n-2}~\text{ with }~n\geq2.
\end{eqnarray*}
Fibonacci numbers have many interesting properties and applications in arts,
sciences and architecture. For example, the ratio sequences of Fibonacci
numbers converges to the golden ratio which is important in sciences and
arts. Also, some basic properties of sequences of Fibonacci numbers \cite{29}
are given as follows: 
\begin{eqnarray*}
&&\lim_{n\to\infty}\frac{f_{n+1}}{f_{n}}=\frac{1+\sqrt{5}}{2}=\varphi \text{
\ (Golden Ratio)}, \\
&&\sum_{k=0}^{n}f_{k}=f_{n+2}-1~\text{ for each }~n\in\mathbb{N}, \\
&&\sum_{k}\frac{1}{f_{k}}~\text{ converges}, \\
&&f_{n-1}f_{n+1}-f_{n}^{2}=(-1)^{n+1}~\text{ for all }n\geq1\text{ (Cassini
Formula)}.
\end{eqnarray*}
One can easily derive by substituting $f_{n+1}$ in Cassini's formula that $%
f_{n-1}^{2}+f_{n}f_{n-1}-f_{n}^{2}=(-1)^{n+1}$.

Now, let $A=(a_{nk})$ be an infinite matrix and list the following
conditions: 
\begin{eqnarray}
&&\sup_{n\in\mathbb{N}}\sum_{k}\left \vert a_{nk}\right \vert<\infty
\label{eq21} \\
&&\lim_{n\to\infty}a_{nk}=0~\text{ for each }~k\in\mathbb{N}  \label{eq22} \\
&&\exists\alpha_{k}\in\mathbb{C}\ni\lim_{n\to\infty}a_{nk}=\alpha_{k}~\text{
for each }~k\in\mathbb{N}  \label{eq23} \\
&&\lim_{n\to\infty}\sum_{k}a_{nk}=0  \label{eq24} \\
&&\exists\alpha\in\mathbb{C}\ni\lim_{n\to\infty}\sum_{k}a_{nk}=\alpha
\label{eq25} \\
&&\sup_{K\in\mathcal{F}}\sum_{n}\left\vert\sum_{k\in K}a_{nk}\right\vert
<\infty,  \label{eq26}
\end{eqnarray}
where $\mathbb{C}$ and $\mathcal{F}$ denote the set of all complex numbers
and the collection of all finite subsets of $\mathbb{N}$, respectively.

Now, we may give the following lemma due to Stieglitz and Tietz \cite{30} on
the characterization of the matrix transformations between some classical
sequence spaces.

\begin{lem}
\label{l21}\label{l21} The following statements hold:

\begin{enumerate}
\item[(a)] $A=(a_{nk})\in(c_{0},c_{0})$ if and only if (\ref{eq21}) and (\ref%
{eq22}) hold.

\item[(b)] $A=(a_{nk})\in(c_{0},c)$ if and only if (\ref{eq21}) and (\ref%
{eq23}) hold.

\item[(c)] $A=(a_{nk})\in(c,c_{0})$ if and only if (\ref{eq21}), (\ref{eq22}%
) and (\ref{eq24}) hold.

\item[(d)] $A=(a_{nk})\in(c,c)$ if and only if (\ref{eq21}), (\ref{eq23})
and (\ref{eq25}) hold.

\item[(e)] $A=(a_{nk})\in(c_{0},\ell_{\infty})=(c,\ell_{\infty})$ if and
only if condition (\ref{eq21}) holds.

\item[(f)] $A=(a_{nk})\in(c_{0},\ell_{1})=(c,\ell_{1})$ if and only if
condition (\ref{eq26}) holds.
\end{enumerate}
\end{lem}

\section{The Fibonacci Difference Spaces of Null and Convergent Sequences}

In this section, we define the spaces $c_{0}(\widehat{F})$ and $c(\widehat{F}%
)$ of Fibonacci null and Fibonacci convergent sequences. Also, we present
some inclusion theorems and construct the Schauder bases of the spaces $%
c_{0}(\widehat{F})$ and $c(\widehat{F})$.

Recently, Kara \cite{28} has defined the sequence space $\ell_{p}(\widehat{F}%
)$ as follows: 
\begin{eqnarray*}
\ell_{p}(\widehat{F})=\left\{x\in \omega:\widehat{F}x\in\ell_{p}\right%
\},~~(1\leq p\leq\infty),
\end{eqnarray*}
where $\widehat{F}=(\widehat{f}_{nk})$ is the double band matrix defined by
the sequence $(f_{n})$ of Fibonacci numbers as follows 
\begin{eqnarray}  \label{eqy31}
\widehat{f}_{nk}=\left\{%
\begin{array}{ccl}
-\frac{f_{n+1}}{f_{n}} & , & k=n-1, \\ 
\frac{f_{n}}{f_{n+1}} & , & k=n, \\ 
0 & , & 0\leq k<n-1\text{ or }k>n%
\end{array}
\right.
\end{eqnarray}
for all $k,n\in\mathbb{N}$. Also, in \cite{31}, Kara et al. have
characterized some classes of compact operators on the spaces $\ell_{p}(%
\widehat{F})$ and $\ell_{\infty}(\widehat {F})$, where $1\leq p<\infty.$

One can derive by a straightforward calculation that the inverse $\widehat {F%
}^{-1}=(g_{nk})$ of the Fibonacci matrix $\widehat{F}$ is given by 
\begin{eqnarray*}
g_{nk}=\left\{%
\begin{array}{ccl}
\frac{f_{n+1}^{2}}{f_{k}f_{k+1}} & , & 0\leq k\leq n, \\ 
0 & , & k>n%
\end{array}
\right.
\end{eqnarray*}
for all $k,n\in\mathbb{N}$.

Now, we introduce the Fibonacci difference sequence spaces $c_{0}(\widehat{F}%
)$ and $c(\widehat{F})$ as the set of all sequences whose $\widehat{F}$%
-transforms are in the spaces $c_{0}$ and $c$, respectively, i.e., 
\begin{eqnarray*}
c_{0}(\widehat{F})&=&\left \{x=(x_{n})\in \omega:\lim_{n\to\infty}\left( 
\frac{f_{n}}{f_{n+1}}x_{n}-\frac{f_{n+1}}{f_{n}}x_{n-1}\right)=0\right\}, \\
c(\widehat{F})&=&\left \{x=(x_{n})\in \omega:\exists l\in\mathbb{C}%
\ni\lim_{n\to\infty}\left(\frac {f_{n}}{f_{n+1}}x_{n}-\frac{f_{n+1}}{f_{n}}%
x_{n-1}\right)=l\right\}.
\end{eqnarray*}
With the notation of (\ref{eq12}), the spaces $c_{0}(\widehat{F})$ and $c(%
\widehat{F})$ can be redefined as follows: 
\begin{eqnarray*}  \label{eq31}
c_{0}(\widehat{F})=(c_{0})_{\widehat{F}}\text{ \ and \ }c(\widehat {F})=c_{%
\widehat{F}}.
\end{eqnarray*}

Define the sequence $y=(y_{n})$ by the $\widehat{F}$-transform of a sequence 
$x=(x_{n})$, i.e., 
\begin{eqnarray}  \label{eq32}
y_{n}=\widehat{F}_{n}(x)=\left\{%
\begin{array}{ccl}
x_{0} & , & n=0, \\ 
\frac{f_{n}}{f_{n+1}}x_{n}-\frac{f_{n+1}}{f_{n}}x_{n-1} & , & n\geq1%
\end{array}%
\right.
\end{eqnarray}%
for all $n\in\mathbb{N}$. Throughout the text, we suppose that the sequences 
$x=(x_{k})$ and $y=(y_{k})$ are connected with the relation (\ref{eq32}).

\begin{thm}
The sets $c_{0}(\widehat{F})$ and $c(\widehat{F})$ are the linear spaces
with the co-ordinatewise addition and scalar multiplication which are the
BK-spaces with the norm $\left \Vert x\right \Vert _{c_{0}(\widehat{F}%
)}=\left \Vert x\right \Vert _{c(\widehat{F})}=\Vert \widehat{F}%
x\Vert_{_{\infty}}$.
\end{thm}

\begin{proof}
This is a routine verification and so we omit the detail.
\end{proof}

\begin{rem}
One can easily check that the absolute property does not hold on the spaces $%
c_{0}(\widehat{F})$ and $c(\widehat{F})$, that is $\left \Vert x\right \Vert
_{c_{0}(\widehat{F})}\neq \left \Vert \left \vert x\right \vert \right \Vert
_{c_{0}(\widehat{F})}$ and $\left \Vert x\right \Vert _{c(\widehat{F})} \neq
\left \Vert \left \vert x\right \vert \right \Vert _{c(\widehat{F})}$ for at
least one sequence in the spaces $c_{0}(\widehat{F})$ and $c(\widehat{F})$,
and this says us that $c_{0}(\widehat{F})$ and $c(\widehat{F})$ are the
sequence spaces of non-absolute type, where $\left \vert x\right \vert
=(\left \vert x_{k}\right \vert ).$
\end{rem}

\begin{thm}
The Fibonacci difference sequence spaces $c_{0}(\widehat{F})$ and $c(%
\widehat{F})$ of non-absolute type are linearly norm isomorphic to the
spaces $c_{0}$ and $c$, respectively, i.e., $c_{0}(\widehat{F})\cong c_{0}$
and $c(\widehat{F})\cong c.$
\end{thm}

\begin{proof}
Since the proof is similar for the space $c(\widehat{F})$, we consider only
the space $c_{0}(\widehat{F})$. With the notation of (\ref{eq32}), we
consider the transformation $T$ defined from $c_{0}(\widehat{F})$ to $c_{0}$
by $x\mapsto y=\widehat{F}x$. It is trivial that the map $T$ is linear and
injective. Furthermore, let $y=(y_{k})\in c_{0}$ and define the sequence $%
x=(x_{k})$ by 
\begin{eqnarray}  \label{eq33}
x_{k}=\sum_{j=0}^{k}\frac{f_{k+1}^{2}}{f_{j}f_{j+1}}y_{j}~\text{ for all }%
~k\in\mathbb{N}.
\end{eqnarray}
Then we have 
\begin{eqnarray*}
\lim_{k\to\infty}\widehat{F}_{k}(x)=\lim_{k\to\infty}\left( \frac{f_{k}}{%
f_{k+1}}\sum_{j=0}^{k}\frac{f_{k+1}^{2}} {f_{j}f_{j+1}}y_{j}-\frac{f_{k+1}}{%
f_{k}}\sum_{j=0}^{k-1}\frac{f_{k}^{2} }{f_{j}f_{j+1}}y_{j}\right)
=\lim_{k\to\infty}y_{k}=0
\end{eqnarray*}
which says us that $x\in c_{0}(\widehat{F}).$ Additionally, we have for
every $x\in c_{0}(\widehat{F})$ that 
\begin{eqnarray*}
\left \Vert x\right \Vert _{c_{0}(\widehat{F})}=\sup_{k\in\mathbb{N}}\left
\vert \frac{f_{k} }{f_{k+1}}x_{k}-\frac{f_{k+1}}{f_{k}}x_{k-1}\right \vert
=\sup_{k\in\mathbb{N}}\left \vert y_{k}\right \vert =\left \Vert y\right
\Vert _{\infty}<\infty.
\end{eqnarray*}
Consequently we see from here that $\widehat{F}$ is surjective and norm
preserving. Hence, $\widehat{F}$ is a linear bijection which shows that the
spaces $c_{0}(\widehat{F})$ and $c_{0}$ are linearly isomorphic. This
concludes the proof.
\end{proof}

Now, we give some inclusion relations concerning with the spaces $c_{0}(%
\widehat{F})$ and $c(\widehat{F})$.

\begin{thm}
The inclusion $c_{0}(\widehat{F})\subset c(\widehat{F})$ strictly holds.
\end{thm}

\begin{proof}
It is clear that the inclusion $c_{0}(\widehat{F})\subset c(\widehat{F})$
holds. Further, to show that this inclusion is strict, consider the sequence 
$x=(x_{k})=\left(\sum_{j=0}^{k}f_{k+1}^{2}/f_{j}^{2}\right)$. Then, we
obtain by (\ref{eq32}) for all $k\in\mathbb{N}$ that 
\begin{eqnarray*}
\widehat{F}_{k}(x)=\frac{f_{k}}{f_{k+1}}\sum_{j=0}^{k}\frac{f_{k+1}^{2}} {%
f_{j}^{2}}-\frac{f_{k+1}}{f_{k}}\sum_{j=0}^{k-1}\frac{f_{k}^{2}}{f_{j}^{2} }=%
\frac{f_{k+1}}{f_{k}}
\end{eqnarray*}
which shows that $\widehat{F}_{k}(x)=\frac{f_{k+1}}{f_{k}}\to\varphi$, as $%
k\to\infty$. That is to say that $\widehat{F}(x)\in c\setminus c_{0}$. Thus,
the sequence $x$ is in $c(\widehat{F})$ but not in $c_{0}(\widehat{F}).$
Hence, the inclusion $c_{0}(\widehat{F})\subset c(\widehat{F})$ is strict.
\end{proof}

\begin{thm}
The space $\ell_{\infty}$ does not include the spaces $c_{0}(\widehat{F})$
and $c(\widehat{F})$.
\end{thm}

\begin{proof}
Let us consider the sequence $x=(x_{k})=(f_{k+1}^{2})$. Since $%
f_{k+1}^{2}\rightarrow\infty$ as $k\rightarrow \infty$ and $\widehat{F}%
x=e^{(0)}=(1,0,0,\ldots )$, the sequence $x$ is in the space $c_{0}(\widehat{%
F})$ but is not in the space $\ell_{\infty}$. This shows that the space $%
\ell_{\infty}$ does not include both the space $c_{0}(\widehat{F})$ and the
space $c(\widehat{F})$, as desired.
\end{proof}

\begin{thm}
The inclusions $c_{0}\subset c_{0}(\widehat{F})$ and $c\subset c(\widehat{F}%
) $ strictly hold.
\end{thm}

\begin{proof}
Let $\lambda=c_{0}$ or $c$. Since the matrix $\widehat{F}=(f_{nk})$
satisfies the conditions 
\begin{eqnarray*}
&&\sup_{n\in\mathbb{N}}\sum_{k}\left \vert f_{nk}\right \vert =\sup_{n\in%
\mathbb{N}}\left( \frac{f_{n}}{f_{n+1}}+\frac{f_{n+1}}{f_{n}}\right) =2+%
\frac{1}{2}=\frac{5}{2}, \\
&&\lim_{n\to\infty}f_{nk}=0, \\
&& \lim_{n\to\infty}\sum_{k}f_{nk}=\lim_{n\to\infty}\left( \frac{f_{n}}{%
f_{n+1}} -\frac{f_{n+1}}{f_{n}}\right) =\frac{1}{\varphi}-\varphi
\end{eqnarray*}
we conclude by Parts (a) and (c) of Lemma \ref{l21} that $\widehat{F}%
\in(\lambda,\lambda)$. This leads that $\widehat{F}x\in \lambda$ for any $%
x\in \lambda$. Thus, $x\in \lambda_{\widehat{F}}$. This shows that $\lambda
\subset \lambda _{\widehat{F}}$.

Now, let $x=(x_{k})=(f_{k+1}^{2}).$ Then, it is clear that $x\in \lambda_{%
\widehat{F}}\setminus \lambda$. This says that the inclusion $\lambda
\subset \lambda_{\widehat{F}}$ is strict.
\end{proof}

\begin{thm}
The spaces $c_{0}(\widehat{F})$ and $c(\widehat{F})$ are not solid.
\end{thm}

\begin{proof}
Consider the sequences $u=(u_{k})$ and $v=(v_{k})$ defined by $%
u_{k}=f_{k+1}^{2}$ and $v_{k}=(-1)^{k+1}$ for all $k\in\mathbb{N}$. Then, it
is clear that $u\in c_{0}(\widehat{F})$ and $v\in \ell_{\infty}$.
Nevertheless $uv=\{(-1)^{k+1}f_{k+1}^{2}\}$ is not in the space $c_{0}(%
\widehat{F})$, since 
\begin{eqnarray*}
\widehat{F}_{k}(uv)=\frac{f_{k}}{f_{k+1}}(-1)^{k+1}f_{k+1}^{2}-\frac{f_{k+1}%
}{f_{k}}(-1)^{k}f_{k}^{2}=2(-1)^{k+1}f_{k}f_{k+1}
\end{eqnarray*}%
for all $k\in\mathbb{N}$. This shows that the multiplication $%
\ell_{\infty}c_{0}(\widehat{F})$ of the spaces $\ell_{\infty}$ and $c_{0}(%
\widehat{F})$ is not a subset of $c_{0}(\widehat{F})$. Hence, the space $%
c_{0}(\widehat{F})$ is not solid.

It is clear here that if the spaces $c_{0}(\widehat{F})$ is replaced by the
space $c(\widehat{F}),$ then we obtain the fact that $c(\widehat{F})$ is not
solid. This completes the proof.
\end{proof}

It is known from Theorem 2.3 of Jarrah and Malkowsky \cite{ajm} that the
domain $\lambda_T$ of an infinite matrix $T=(t_{nk})$ in a normed sequence
space $\lambda$ has a basis if and only if $\lambda$ has a basis, if $T$ is
a triangle. As a direct consequence of this fact, we have:

\begin{cor}
Define the sequences $c^{(-1)}=\big\{c_{k}^{(-1)}\big\}_{k\in\mathbb{N}}$
and $c^{(n)}=\big\{c_{k}^{(n)}\big\}_{k\in\mathbb{N}}$ for every fixed $n\in%
\mathbb{N}$ by 
\begin{eqnarray*}
c_{k}^{(-1)}=\sum_{j=0}^{k}\frac{f_{k+1}^{2}}{f_{j}f_{j+1}}\text{ \ and \ }%
c_{k}^{(n)}=\left\{%
\begin{array}{ccl}
0 & , & 0\leq k\leq n-1 \\ 
\frac{f_{k+1}^{2}}{f_{n}f_{n+1}} & , & k\geq n.%
\end{array}%
\right.
\end{eqnarray*}
Then, the following statements hold:

\begin{enumerate}
\item[(a)] The sequence $\left\{c^{(n)}\right\}_{n=0}^{\infty}$ is a basis
for the space $c_{0}(\widehat{F})$ and every sequence $x\in c_{0}(\widehat{F}%
)$ has a unique representation $x=\sum_{n}\widehat{F}_{n}(x)c^{(n)}$.

\item[(b)] The sequence $\left\{c^{(n)}\right\}_{n=-1}^{\infty}$ is a basis
for the space $c(\widehat{F})$ and every sequence $z=(z_{n})\in c(\widehat{F}%
)$ has a unique representation $z=lc^{(-1)}+\sum_{n}\left[\widehat{F}%
_{n}(z)-l\right]c^{(n)}$, where $l=\lim_{n\to\infty}\widehat{F}_{n}(z)$.
\end{enumerate}
\end{cor}

\section{The alpha-, beta- and gamma-duals of the spaces $c_{0}(\widehat{F})$
and $c(\widehat{F})$, and some matrix transformations}

In this section, we determine the alpha-, beta- and gamma-duals of the
spaces $c_{0}(\widehat{F})$ and $c(\widehat{F})$, and characterize the
classes of infinite matrices from the spaces $c_{0}(\widehat{F})$ and $c(%
\widehat{F})$ to the spaces $c_{0}$, $c$, $\ell_{\infty}$, $f$, $f_{0}$, $bs$%
, $fs$, $cs$ and $\ell_{1}$, and from the space $f$ to the spaces $c_{0}(%
\widehat{F})$ and $c(\widehat{F})$.

Now, we quote two lemmas required in proving the theorems concerning the
alpha-, beta- and gamma-duals of the spaces $c_{0}(\widehat{F})$ and $c(%
\widehat{F})$.

\begin{lem}
Let $\lambda$ be any of the spaces $c_{0}$ or $c$ and $a=(a_{n})\in\omega$,
and the matrix $B=(b_{nk})$ be defined by $B_{n}=a_{n}\widehat{F}_{n}^{-1}$,
that is, 
\begin{eqnarray*}
b_{nk}=\left\{%
\begin{array}{ccl}
a_{n}g_{nk} & , & 0\leq k\leq n, \\ 
0 & , & k>n%
\end{array}
\right.
\end{eqnarray*}
for all $k,n\in\mathbb{N}$. Then, $a\in \lambda_{\widehat{F}}^{\beta}$ if
and only if $B\in(\lambda,\ell_{1})$.
\end{lem}

\begin{proof}
Let $y$ be the $\widehat{F}$-transform of a sequence $x=(x_{n})\in\omega$.
Then, we have by (\ref{eq33}) that 
\begin{eqnarray}  \label{eq41}
a_{n}x_{n}=a_{n}\widehat{F}_{n}^{-1}(y)=B_{n}(y)~\text{ for all }~n\in%
\mathbb{N}.
\end{eqnarray}
Thus, we observe by (\ref{eq41}) that $ax=(a_{n}x_{n})\in \ell_{1}$ with $%
x\in\lambda_{\widehat{F}}$ if and only if $By\in\ell_{1}$ with $y\in\lambda$%
. This means that $a\in\lambda_{\widehat{F}}^{\beta}$ if and only if $%
B\in(\lambda,\ell_{1})$. This completes the proof.
\end{proof}

\begin{lem}
\emph{\cite[Theorem 3.1]{33}} Let $C=(c_{nk})$ be defined via a sequence $%
a=(a_{k})\in \omega$ and the inverse matrix $V=(v_{nk})$ of the triangle
matrix $U=(u_{nk})$ by 
\begin{eqnarray*}
c_{nk}=\left \{%
\begin{array}{ccl}
\sum_{j=k}^{n}a_{j}v_{jk} & , & 0\leq k\leq n, \\ 
0 & , & k>n%
\end{array}
\right.
\end{eqnarray*}
for all $k,n\in\mathbb{N}$. Then, for any sequence space $\lambda$, 
\begin{eqnarray*}
\lambda_{U}^{\gamma}&=&\left\{a=(a_{k})\in
\omega:C\in(\lambda,\ell_{\infty})\right\}, \\
\lambda_{U}^{\beta}&=&\left\{a=(a_{k})\in \omega:C\in(\lambda,c)\right\}.
\end{eqnarray*}
\end{lem}

Combining Lemmas 2.1, 4.1 and 4.2, we have;

\begin{cor}
Consider the sets $d_{1}$, $d_{2}$, $d_{3}$ and $d_{4}$ defined as follows: 
\begin{eqnarray*}
d_{1}&=&\left\{a=(a_{k})\in \omega:\sup_{K\in\mathcal{F}}\sum_{n}\left\vert%
\sum_{k\in K}\frac{f_{n+1}^{2}}{f_{k}f_{k+1}}a_{n}\right\vert<\infty\right\},
\\
d_{2}&=&\left\{a=(a_{k})\in\omega:\sup_{n\in\mathbb{N}}\sum_{k=0}^{n}\left%
\vert\sum\limits_{j=k}^{n}\frac{f_{j+1}^{2}}{f_{k}f_{k+1}}%
a_{j}\right\vert<\infty\right\}, \\
d_{3}&=&\left\{ a=(a_{k})\in \omega:\lim_{n\to\infty}\sum_{j=k}^{n}\frac{%
f_{j+1}^{2}}{f_{k}f_{k+1}}a_{j}\text{ exists for each }k\in\mathbb{N}%
\right\}, \\
d_{4}&=&\left\{a=(a_{k})\in
\omega:\lim_{n\to\infty}\sum_{k=0}^{n}\sum_{j=k}^{n}\frac{f_{j+1}^{2}}{%
f_{k}f_{k+1}}a_{j} \text{ exists}\right\} .
\end{eqnarray*}
Then, the following statements hold:

\begin{enumerate}
\item[(a)] $\big\{c_{0}(\widehat{F})\big\}^{\alpha}=\big\{c(\widehat{F})%
\big\}^{\alpha}=d_{1}$.

\item[(b)] $\big\{c_{0}(\widehat{F})\big\}^{\beta}=d_{2}\cap d_{3}$ and $%
\big\{c(\widehat{F})\big\}^{\beta}=d_{2}\cap d_{3}\cap d_{4}$.

\item[(c)] $\big\{c_{0}(\widehat{F})\big\}^{\gamma}=\big\{c(\widehat{F})%
\big\}^{\gamma}=d_{2}$.
\end{enumerate}
\end{cor}

\begin{thm}
\label{th44} Let $\lambda=c_{0}$ or $c$ and $\mu$ be an arbitrary subset of $%
\omega$. Then, we have $A=(a_{nk})\in(\lambda_{\widehat{F}},\mu)$ if and
only if 
\begin{eqnarray}  \label{eq42}
&&D^{(m)}=\left(d_{nk}^{(m)}\right)\in(\lambda,c)\text{ for all }n\in\mathbb{%
N}, \\
&&D=(d_{nk})\in(\lambda,\mu),  \label{eq43}
\end{eqnarray}
where $d_{nk}^{(m)}=\left\{%
\begin{array}{ccl}
\sum_{j=k}^{m}\frac{f_{j+1}^{2}}{f_{k}f_{k+1}}a_{nj} & , & 0\leq k\leq m, \\ 
0 & , & k>m%
\end{array}%
\right.$ and $d_{nk}=\sum_{j=k}^{\infty}\frac{f_{j+1}^{2}}{f_{k}f_{k+1}}%
a_{nj}$ for all $k,m,n\in\mathbb{N}$.
\end{thm}

\begin{proof}
To prove the theorem, we follow the similar way due to Kiri\c{s}\c{c}i and Ba%
\c{s}ar \cite{4}. Let $A=(a_{nk})\in(\lambda_{\widehat{F}},\mu)$ and $%
x=(x_{k})\in\lambda_{\widehat{F}}$. Recalling the relation (\ref{eq33}) we
have 
\begin{eqnarray}  \label{eq44}
\sum_{k=0}^{m}a_{nk}x_{k}&=&\sum_{k=0}^{m}a_{nk}\sum_{j=0}^{k}\frac{%
f_{k+1}^{2}}{f_{j}f_{j+1}}y_{j} \\
&=&\sum_{k=0}^{m}\sum_{j=k}^{m}\frac{f_{j+1}^{2}}{f_{k}f_{k+1}}%
a_{nj}y_{k}=\sum_{k=0}^{m}d_{nk}^{(m)}y_{k}=D_{n}^{(m)}(y)  \notag
\end{eqnarray}
for all $m,n\in\mathbb{N}$. Since $Ax$ exists, $D^{(m)}$ must belong to the
class $(\lambda,c)$. Letting $m\rightarrow \infty$ in the equality (\ref%
{eq44}), we obtain $Ax=Dy$ which gives the result $D\in(\lambda,\mu)$.

Conversely, suppose the conditions (\ref{eq42}), (\ref{eq43}) hold and take
any $x\in\lambda_{\widehat{F}}$. Then, we have $(d_{nk})_{k\in\mathbb{N}%
}\in\lambda^{\beta}$ which gives together with (\ref{eq42}) that $%
A_{n}=(a_{nk})_{k\in\mathbb{N}}\in\lambda_{\widehat{F}}^{\beta}$ for all $%
n\in\mathbb{N}$. Thus, $Ax$ exists. Therefore, we derive by the equality (%
\ref{eq44}) as $m\rightarrow \infty$ that $Ax=Dy,$ and this shows that $%
A\in(\lambda_{\widehat{F}},\mu)$.
\end{proof}

By changing the roles of the spaces $\lambda_{\widehat{F}}$ and $\lambda$
with $\mu$ in Theorem \ref{th44}, we have:

\begin{thm}
\label{th45} Suppose that the elements of the infinite matrices $A=(a_{nk})$
and $B=(b_{nk})$ are connected with the relation 
\begin{eqnarray}  \label{eql45}
b_{nk}:=-\frac{f_{n+1}}{f_{n}}a_{n-1,k}+\frac{f_{n}}{f_{n+1}}a_{nk}
\end{eqnarray}%
for all $k,n\in\mathbb{N}$ and $\mu$ be any given sequence space. Then, $%
A\in(\mu,\lambda_{\widehat{F}})$ if and only if $B\in(\mu,\lambda)$.
\end{thm}

\begin{proof}
Let $z=(z_k)\in\mu$. Then, by taking into account the relation (\ref{eql45})
one can easily derive the following equality 
\begin{eqnarray*}
\sum_{k=0}^mb_{nk}z_{k}=\sum_{k=0}^m\left(-\frac{f_{n+1}}{f_{n}}a_{n-1,k}+%
\frac{f_{n}}{f_{n+1}}a_{nk}\right)z_{k}~\text{ for all }~m,n\in \mathbb{N}
\end{eqnarray*}%
which yields as $m\to\infty$ that $(Bz)_n=[\widehat{F}(Az)]_n$. Therefore,
we conclude that $Az\in\lambda_{\widehat{F}}$ whenever $z\in\mu$ if and only
if $Bz\in\lambda$ whenever $z\in\mu$.

This step completes the proof.
\end{proof}

Prior to giving some natural consequences of Theorems \ref{th44} and \ref%
{th45}, we give the concept of \textit{almost convergence}. The \textit{%
shift operator} $P$ is defined on $\omega $ by $P_n(x)=x_{n+1}$ for all $n\in%
\mathbb{N}$. A Banach limit $L$ is defined on $\ell_\infty$, as a
non-negative linear functional, such that $L(Px)=L(x)$ and $L(e)=1$. A
sequence $x=(x_k)\in\ell_\infty$ is said to be almost convergent to the
generalized limit $l$ if all Banach limits of $x$ coincide and are equal to $%
l$ \cite{lor}, and is denoted by $f-\lim x_k=l$. Let us define $t_{mn}(x)$
via a sequence $x=(x_k)$ by 
\begin{eqnarray*}
t_{mn}(x)=\frac{1}{m+1}\sum_{k=0}^mx_{n+k}~\text{ for all }~m,n\in\mathbb{N}.
\end{eqnarray*}
Lorentz \cite{lor} proved that $f-\lim x_k=l$ if and only if $%
\lim_{m\to\infty}t_{mn}(x)=l$, uniformly in $n$. It is well-known that a
convergent sequence is almost convergent such that its ordinary and
generalized limits are equal. By $f_{0}$, $f$ and $fs$, we denote the spaces
of almost null and almost convergent sequences and series, respectively.
Now, we can give the following two lemmas characterizing the strongly and
almost conservative matrices:

\begin{lem}
\label{lem46}\emph{\cite{si}} $A=(a_{nk})\in(f,c)$ if and only if (\ref{eq21}%
), (\ref{eq23}) and (\ref{eq25}) hold, and 
\begin{eqnarray}  \label{eq46}
\lim_{n\to\infty}\sum_{k}\Delta(a_{nk}-\alpha_{k})=0  \label{eq49}
\end{eqnarray}%
also holds, where $\Delta(a_{nk}-\alpha_{k})=a_{nk}-\alpha_{k}-(a_{n,k+1}-%
\alpha_{k+1})$ for all $k,n\in\mathbb{N}$.
\end{lem}

\begin{lem}
\label{lem47}\emph{\cite{jpk}} $A=(a_{nk})\in(c,f)$ if and only if (\ref%
{eq21}) holds, and 
\begin{eqnarray}
&&\exists\alpha_{k}\in\mathbb{C}\ni f-\lim~a_{nk}=\alpha_{k}~\text{ for each
fixed } ~k\in\mathbb{N};  \label{eq47} \\
&&\exists\alpha\in\mathbb{C}\ni f-\lim\sum_{k}a_{nk}=\alpha.  \label{eq48}
\end{eqnarray}
\end{lem}

Now, we list the following conditions; 
\begin{eqnarray}
&&\sup_{m\in\mathbb{N}}\sum_{k=0}^{m}\left \vert d_{mk}^{(n)}\right \vert
<\infty  \label{eq45} \\
&&\exists d_{nk}\in\mathbb{C}\ni\lim_{m\to\infty}d_{mk}^{(n)}=d_{nk}\text{
for each }k,n\in\mathbb{N}  \label{eq46} \\
&&\sup_{n\in\mathbb{N}}\sum_{k}\left \vert d_{nk}\right \vert <\infty
\label{eq47} \\
&&\exists\alpha_{k}\in\mathbb{C}\ni\lim_{n\to\infty}d_{nk}=\alpha_{k}\text{
for each }k\in\mathbb{N}  \label{eq48} \\
&&\sup_{N,K\in\mathcal{F}}\left\vert\sum_{n\in N}\sum_{k\in K}d_{nk}\right
\vert <\infty  \label{eq49} \\
&&\exists\beta_{n}\in\mathbb{C}\ni\lim_{m\to\infty}%
\sum_{k=0}^{m}d_{mk}^{(n)}=\beta_{n}\text{ for each }n\in\mathbb{N}
\label{eq410} \\
&&\exists\alpha\in\mathbb{C}\ni\lim_{n\to\infty}\sum_{k}d_{nk}=\alpha
\label{eq411}
\end{eqnarray}

It is trivial that Theorem \ref{th44} and Theorem \ref{th45} have several
consequences. Indeed, combining Theorems \ref{th44}, \ref{th45} and Lemmas %
\ref{l21}, \ref{lem46} and \ref{lem47}, we derive the following results:

\begin{cor}
Let $A=(a_{nk})$ be an infinite matrix and $a(n,k)=\sum_{j=0}^{n}a_{jn}$ for
all $k,n\in\mathbb{N}$. Then, the following statements hold:

\begin{enumerate}
\item[(a)] $A=(a_{nk})\in(c_{0}(\widehat{F}),c_{0})$ if and only if (\ref%
{eq45}), (\ref{eq46}), (\ref{eq47}) hold and (\ref{eq48}) also holds with $%
\alpha_{k}=0$ for all $k\in\mathbb{N}$.

\item[(b)] $A=(a_{nk})\in(c_{0}(\widehat{F}),cs_{0})$ if and only if (\ref%
{eq45}), (\ref{eq46}), (\ref{eq47}) hold and (\ref{eq48}) also holds with $%
\alpha_{k}=0$ for all $k\in\mathbb{N}$ with $a(n,k)$ instead of $a_{nk}$.

\item[(c)] $A=(a_{nk})\in(c_{0}(\widehat{F}),c)$ if and only if (\ref{eq45}%
), (\ref{eq46}), (\ref{eq47}) and (\ref{eq48}) hold.

\item[(d)] $A=(a_{nk})\in(c_{0}(\widehat{F}),cs)$ if and only if (\ref{eq45}%
), (\ref{eq46}), (\ref{eq47}) and (\ref{eq48}) hold with $a(n,k)$ instead of 
$a_{nk}$.

\item[(e)] $A=(a_{nk})\in(c_{0}(\widehat{F}),\ell_{\infty})$ if and only if
conditions (\ref{eq45}), (\ref{eq46}) and (\ref{eq47}) hold.

\item[(f)] $A=(a_{nk})\in(c_{0}(\widehat{F}),bs)$ if and only if conditions (%
\ref{eq45}), (\ref{eq46}) and (\ref{eq47}) hold with $a(n,k)$ instead of $%
a_{nk}$.

\item[(g)] $A=(a_{nk})\in(c_{0}(\widehat{F}),\ell_{1})$ if and only if (\ref%
{eq45}), (\ref{eq46}) and (\ref{eq49}) hold.

\item[(h)] $A=(a_{nk})\in(c_{0}(\widehat{F}),bv_{1})$ if and only if (\ref%
{eq45}), (\ref{eq46}) and (\ref{eq49}) hold with $a_{nk}-a_{n-1,k}$ instead
of $a_{nk}$.
\end{enumerate}
\end{cor}

\begin{cor}
Let $A=(a_{nk})$ be an infinite matrix. Then, the following statements hold:

\begin{enumerate}
\item[(a)] $A=(a_{nk})\in(c(\widehat{F}),\ell_{\infty})$ if and only if (\ref%
{eq45}), (\ref{eq46}), (\ref{eq47}) and (\ref{eq410}) hold.

\item[(b)] $A=(a_{nk})\in(c(\widehat{F}),bs)$ if and only if (\ref{eq45}), (%
\ref{eq46}), (\ref{eq47}) and (\ref{eq410}) hold with $a(n,k)$ instead of $%
a_{nk}$.

\item[(c)] $A=(a_{nk})\in(c(\widehat{F}),c)$ if and only if (\ref{eq45}), (%
\ref{eq46}), (\ref{eq47}), (\ref{eq48}), (\ref{eq410}) and (\ref{eq411})
hold.

\item[(d)] $A=(a_{nk})\in(c(\widehat{F}),cs)$ if and only if (\ref{eq45}), (%
\ref{eq46}), (\ref{eq47}), (\ref{eq48}), (\ref{eq410}) and (\ref{eq411})
hold with $a(n,k)$ instead of $a_{nk}$.

\item[(e)] $A=(a_{nk})\in(c(\widehat{F}),c_{0})$ if and only if (\ref{eq45}%
), (\ref{eq46}), (\ref{eq47}), (\ref{eq48}) hold with $\alpha_{k}=0$ for all 
$k\in\mathbb{N}$, (\ref{eq410}) and (\ref{eq411}) also hold with $\alpha=0$.

\item[(f)] $A=(a_{nk})\in(c(\widehat{F}),cs_{0})$ if and only if (\ref{eq45}%
), (\ref{eq46}), (\ref{eq47}), (\ref{eq48}) hold with $\alpha_{k}=0$ for all 
$k\in\mathbb{N}$, (\ref{eq410}) and (\ref{eq411}) also hold with $\alpha=0$
with $a(n,k)$ instead of $a_{nk}$.

\item[(g)] $A=(a_{nk})\in(c(\widehat{F}),\ell_{1})$ if and only if (\ref%
{eq45}), (\ref{eq46}), (\ref{eq49}) and (\ref{eq410}) hold.

\item[(h)] $A=(a_{nk})\in(c(\widehat{F}),bv_{1})$ if and only if (\ref{eq45}%
), (\ref{eq46}), (\ref{eq49}) and (\ref{eq410}) hold with $a_{nk}-a_{n-1,k}$
instead of $a_{nk}$.
\end{enumerate}
\end{cor}

\begin{cor}
$A=(a_{nk})\in(c(\widehat{F}),f)$ if and only if (\ref{eq45}), (\ref{eq46}),
(\ref{eq410}) and (\ref{eq411}) hold, and (\ref{eq47}), (\ref{eq48}) also
hold with $d_{nk}$ instead of $a_{nk}$.
\end{cor}

\begin{cor}
$A=(a_{nk})\in(c(\widehat{F}),f_{0})$ if and only if (\ref{eq45}), (\ref%
{eq46}), (\ref{eq410}) and (\ref{eq411}) hold, and (\ref{eq47}), (\ref{eq48}%
) also hold with $d_{nk}$ instead of $a_{nk}$ and $\alpha_{k}=0$ for all $%
k\in\mathbb{N}$.
\end{cor}

\begin{cor}
$A=(a_{nk})\in(c(\widehat{F}),fs)$ if and only if (\ref{eq45}), (\ref{eq46}%
), (\ref{eq47}),(\ref{eq48}), (\ref{eq410}) and (\ref{eq411}) hold with $%
a(n,k)$ instead of $a_{nk}$ and (\ref{eq47}), (\ref{eq48}) hold with $d(n,k)$
instead of $d_{nk}$.
\end{cor}

\begin{cor}
$A=(a_{nk})\in(f,c(\widehat{F}))$ if and only if (\ref{eq21}), (\ref{eq23}),
(\ref{eq25}) and (\ref{eq46}) hold with $b_{nk}$ instead of $a_{nk}$, where $%
b_{nk}$ is defined by (\ref{eql45}).
\end{cor}

\begin{cor}
$A=(a_{nk})\in(f,c_{0}(\widehat{F}))$ if and only if (\ref{eq21}) and (\ref%
{eq25}) hold, (\ref{eq23}) and (\ref{eq46}) also hold with $b_{nk}$ instead
of $a_{nk}$ and $\alpha_{k}=0$ for all $k\in\mathbb{N}$, where $b_{nk}$ is
defined by (\ref{eql45}).
\end{cor}

\end{document}